\pdfoutput=1
\documentclass[11pt,a4paper]{article}
\usepackage[latin1]{inputenc}
\usepackage{amsmath}
\usepackage{amsthm}
\usepackage{amsfonts}
\usepackage{amsfonts,amsthm,latexsym,amsmath,amssymb,amscd,epsfig,psfrag,enumerate}
\usepackage{graphics,graphicx, bezier, float, color, hyperref}
\usepackage{amssymb,url}
\usepackage{multienum}
\usepackage[table]{xcolor}
\usepackage{multicol,multirow}
\usepackage{graphicx}
\usepackage{fancyvrb}
\usepackage{parskip}
\usepackage[toc,page]{appendix}
\usepackage[top=2.7cm, bottom=2.7cm, left=1.5cm, right=1.5cm]{geometry}
\usepackage{xcolor}

\usepackage{blkarray}
\newtheorem{theorem}{Theorem}[section]
\newtheorem{lemma}{Lemma}[section]

\usepackage[none]{hyphenat}[section]
\newtheorem{definition}{Definition}[section]

\numberwithin{equation}{section}
\numberwithin{table}{section}
\numberwithin{figure}{section}

\title{On the Largest Prime factor of the $k$--generalized Lucas numbers}
\author{Herbert Batte$^{1,*} $ and Florian Luca$^{2}$}
\date{}

\begin{document}
\maketitle
\abstract{ Let $(L_n^{(k)})_{n\geq 2-k}$ be the sequence of $k$--generalized Lucas numbers for some fixed integer $k\ge 2$ whose first $k$ terms are $0,\ldots,0,2,1$ and each term afterwards is the sum of the preceding $k$ terms. For an integer $m$, let $P(m)$ denote the largest prime factor of $m$, with $P(0)=P(\pm 1)=1$. We show that if $n \ge k + 1$, then $P (L_n^{(k)} ) > (1/86) \log \log n$. Furthermore, we determine all the $k$--generalized Lucas numbers $L_n^{(k)}$ whose largest prime factor is at most $ 7$. } 



\thanks{$ ^{*} $ Corresponding author}

\section{Introduction}\label{intro}
\subsection{Background}\label{sec:1.1}
Let $k \ge 2$ be an integer. The $k$--generalized Lucas numbers  is the sequence defined by the recurrence relation 
$$ L_n^{(k)} = L_{n-1}^{(k)}+\cdots + L_{n-k}^{(k)},~~\text{for all}~~n\ge 2,   $$
with the initial condition $L_0^{(k)}=2$, $L_1^{(k)}=1$ for all $k\ge 2$ and $L_{2-k}^{(k)}=\cdots=L_{-1}^{(k)}=0$ for all $k\ge 3$. When $k = 2$, this sequence is the classical sequence of Lucas numbers and in this case we omit the superscript ${}^{(k)}$ in the notation.  

For an integer $m$,  let $P(m)$ be the largest prime factor of $m$ with the convention $P(0) = P(\pm 1) = 1$. The challenge of determining lower bounds for the largest prime factor of terms in linear recurrence sequences has sparked much interest among mathematicians. Numerous studies have been conducted on this  topic, see, for example \cite{Brl}. In this work, our focus is on the sequence of $k$--generalized Lucas numbers. Specifically, we aim to derive effective lower bounds for $P(L_n^{(k)})$ in relation to both $k$ and $n$.
We prove the following results.
\subsection{Main Results}\label{sec:1.2}
\begin{theorem}\label{1.1} 
	Let $(L_n^{(k)})_{n\geq 2-k}$ be the sequence of $k$--generalized Lucas numbers. Then, the inequality 
	\begin{align*}
		P(L_n^{(k)})>\dfrac{1}{86}\log \log n,
	\end{align*}
holds for all $n\ge k+1$.
\end{theorem}

\begin{theorem}\label{1.2} 
	The only solutions to the Diophantine equation 
	\begin{align}\label{eq1.1}
	L_n^{(k)}=2^a\cdot 3^b\cdot 5^c\cdot 7^d,
	\end{align}
	 in nonnegative integers $n$, $k$, $a$, $b$, $c$, $d$ with $k \ge 2$ and $n \ge k + 1$, are 
	 \begin{align*}
	 &L_3^{(2)}=4,~~~~ L_4^{(2)}=7,~~~~ L_6^{(2)}=18,~~~~ L_4^{(3)}=10,~~~~ L_6^{(3)}=35, \\
	 &L_7^{(3)}=64,~~~~ L_{12}^{(3)}=1350,~~~~ L_{15}^{(3)}=8400,
	 ~~~~ L_8^{(4)}=160 ~~\text{and}~~ L_{15}^{(10)}=24500.\end{align*} 
	 In addition, $L_{n}^{(k)}=3\cdot 2^{n-2}$ for all $2\le n\le k$. 
\end{theorem}

\section{Methods}
\subsection{Preliminaries}
It is known that 
\begin{align}\label{eq2.2}
	L_n^{(k)} = 3 \cdot 2^{n-2},\qquad \text{for all}\qquad 2 \le n \le k.
\end{align}
In particular, $P(L_n^{(k)})=3$ for all $n,k$ in the range $2\le n\le k$. Additionally, $L_{k+1}^{(k)}=3\cdot 2^{k-1}-2$ and by induction one proves that 
\begin{equation}
\label{eq:32}
L_n^{(k)}<3\cdot 2^{n-2}\qquad {\text{\rm holds for all}}\qquad n\ge k+1.
\end{equation}
Next, we revisit some properties of the $k$--generalized Lucas numbers. They form a linearly recurrent sequence of characteristic polynomial
\[
\Psi_k(x) = x^k - x^{k-1} - \cdots - x - 1,
\]
which is irreducible over $\mathbb{Q}[x]$. The polynomial $\Psi_k(x)$ possesses a unique real root $\alpha(k)>1$ and all the other roots are inside the unit circle, see \cite{MIL}. The root  $\alpha(k):=\alpha$ is in the interval
\begin{align}\label{eq2.3}
2(1 - 2^{-k} ) < \alpha < 2
\end{align}
as noted in \cite{WOL}. As in the classical case when $k=2$, it was shown in \cite{BRL} that 
\begin{align}\label{eq2.4}
	\alpha^{n-1} \le L_n^{(k)}\le2\alpha^n, \quad \text{holds for all} \quad n\ge1, ~k\ge 2.
\end{align}
Let $k\ge 2$ and define
\begin{equation}
\label{eq:fk}
	f_k(x):=\dfrac{x-1}{2+(k+1)(x-2)}.
\end{equation}
We have 
$$
\frac{df_k(x)}{dx}=-\frac{(k-1)}{(2+(k+1)(x-2))^2}<0\qquad {\text{\rm for~all}}\qquad x>0.
$$
In particular, inequality \eqref{eq2.3} implies that
\begin{align}\label{eq2.5}
	\dfrac{1}{2}=f_k(2)<f_k(\alpha)<f_k(2(1 - 2^{-k} ))\le \dfrac{3}{4},
\end{align}
for all $k\ge 3$. It is easy to check that the above inequality holds for $k=2$ as well. Further, it is easy to verify that $|f_k(\alpha_i)|<1$, for all $2\le i\le k$, where $\alpha_i$ are the remaining 
roots of $\Psi_k(x)$ for $i=2,\ldots,k$.

The following lemma will be useful in our applications of Baker's theory. It is Lemma 2 in \cite{GoL}.
\begin{lemma}[Lemma 2, \cite{GoL}]\label{lem2.1}
	For all $k \ge 2$, the number $f_k (\alpha)$ is not an algebraic integer.
\end{lemma}
Moreover, it was shown in \cite{BRL} that
\begin{align}\label{eq2.6}
	L_n^{(k)}=\displaystyle\sum_{i=1}^{k}(2\alpha_i-1)f_k(\alpha_i)\alpha_i^{n-1}~~\text{and}~~\left|L_n^{(k)}-f_k(\alpha)(2\alpha-1)\alpha^{n-1}\right|<\dfrac{3}{2},
\end{align}
for all $k\ge 2$ and $n\ge 2-k$. This means that 
\begin{equation}
\label{eq:Lnwitherror}
L_n^{(k)}=f_k(\alpha)(2\alpha-1)\alpha^{n-1}+e_k(n), \qquad {\text{\rm where}}\qquad |e_k(n)|<1.5.
\end{equation} 
The left expression in \eqref{eq2.6} is known as the Binet-like formula for $L_n^{(k)}$. Furthermore, the right inequality expression in \eqref{eq2.6} shows that the contribution of the zeros that are inside the unit circle to $L_n^{(k)}$ is small. 

Next, let $n\ge 3$ and $L_n^{(k)}=p_1^{\beta_1}\cdot p_2^{\beta_2}\cdots p_s^{\beta_s}$ be the prime factorization of the positive integer $L_n^{(k)}$, where $2=p_1<p_2<\cdots <p_s$ is the increasing sequence of prime numbers and the numbers $\beta_i$, for $i=1,2,\ldots, s$, are nonnegative integers. By the right-hand side of relation \eqref{eq2.4}, we have 
\[
\log L_n^{(k)} \le \log 2+n\log\alpha<(n+1)\log2,
\]
since $\alpha<2$, for all $k\ge 2$ in \eqref{eq2.3}. Therefore, we can write
\begin{align*}
	\log\prod_{i=1}^s p_i^{\beta_i} & =\sum_{i=1}^{s}\beta_i\log p_i =\log L_n^{(k)}< (n+1)\log 2,\\	
\sum_{i=1}^{s}\beta_i\log 2 & \le \sum_{i=1}^{s}\beta_i\log p_i <(n+1)\log 2, \quad \text{since}\quad 2\le p_i, \quad \text{for all}\quad i=1,2,\ldots s,
\end{align*}
from which we get
$$
\sum_{i=1}^{s}\beta_i <n+1.
$$
In particular, 
\begin{align}\label{eq2.7}
	\beta_i <n+1, 
\end{align}
for all $i=1,2,\ldots s$. Lastly, for $k \geq 3$ one checks that  $1/\log \alpha\le 2$ by using $\alpha\ge 2(1-1/2^{k})\ge 2(1-1/2^3)$. When $ k = 2$, the number $\alpha$ represents the golden ratio for which  $1/\log\alpha<2.1.$  Thus, the inequality
\begin{align}\label{eq2.8}
\frac{1}{\log \alpha} < 2.1\quad {\text{\rm holds~for~all}}\quad k\ge 2.
\end{align}

\subsection{Linear forms in logarithms}
We use Baker-type lower bounds for nonzero linear forms in logarithms of algebraic numbers. There are many such bounds mentioned in the literature but we use one of Matveev from \cite{MAT}. Before we can formulate such inequalities, we need the notion of height of an algebraic number recalled below.  

\begin{definition}\label{def2.1}
	Let $ \gamma $ be an algebraic number of degree $ d $ with minimal primitive polynomial over the integers $$ a_{0}x^{d}+a_{1}x^{d-1}+\cdots+a_{d}=a_{0}\prod_{i=1}^{d}(x-\gamma^{(i)}), $$ where the leading coefficient $ a_{0} $ is positive. Then, the logarithmic height of $ \gamma$ is given by $$ h(\gamma):= \dfrac{1}{d}\Big(\log a_{0}+\sum_{i=1}^{d}\log \max\{|\gamma^{(i)}|,1\} \Big). $$
\end{definition}
In particular, if $ \gamma$ is a rational number represented as $\gamma=p/q$ with coprime integers $p$ and $ q\ge 1$, then $ h(\gamma ) = \log \max\{|p|, q\} $. 
The following properties of the logarithmic height function $ h(\cdot) $ will be used in the rest of the paper without further reference:
\begin{equation}\nonumber
	\begin{aligned}
		h(\gamma_{1}\pm\gamma_{2}) &\leq h(\gamma_{1})+h(\gamma_{2})+\log 2;\\
		h(\gamma_{1}\gamma_{2}^{\pm 1} ) &\leq h(\gamma_{1})+h(\gamma_{2});\\
		h(\gamma^{s}) &= |s|h(\gamma)  \quad {\text{\rm valid for}}\quad s\in \mathbb{Z}.
	\end{aligned}
\end{equation}
With these properties, it was easily computed in Section 3, equation (12) of \cite{Brl} that
\begin{align}\label{eq2.9}
	h\left(f_k(\alpha)\right)<3\log k, ~~\text{for all}~~k\ge 2.
\end{align}

A linear form in logarithms is an expression
\begin{equation}
	\label{eq:Lambda}
	\Lambda:=b_1\log \gamma_1+\cdots+b_t\log \gamma_t,
\end{equation}
where for us $\gamma_1,\ldots,\gamma_t$ are positive real  algebraic numbers and $b_1,\ldots,b_t$ are integers. We assume, $\Lambda\ne 0$. We need lower bounds 
for $|\Lambda|$. We write ${\mathbb K}:={\mathbb Q}(\gamma_1,\ldots,\gamma_t)$ and $D$ for the degree of ${\mathbb K}$ over ${\mathbb Q}$.
We give Matveev's inequality from \cite{MAT}. 

\begin{theorem}[Matveev, \cite{MAT}]
	\label{thm:Mat} 
	Put $\Gamma:=\gamma_1^{b_1}\cdots \gamma_t^{b_t}-1=e^{\Lambda}-1$. Then 
	$$
	\log |\Gamma|>-1.4\cdot 30^{t+3}\cdot t^{4.5} \cdot D^2 (1+\log D)(1+\log B)A_1\cdots A_t,
	$$
	where $B\ge \max\{|b_1|,\ldots,|b_t|\}$ and $A_i\ge \max\{Dh(\gamma_i),|\log \gamma_i|,0.16\}$ for $i=1,\ldots,t$.
\end{theorem}
During the calculations, upper bounds on the variables are obtained which are too large, thus there is need to reduce them. To do so, we use some results from
approximation lattices and the so--called LLL--reduction method from \cite{LLL}. We explain this in the following subsection.

\subsection{Reduced Bases for Lattices and LLL--reduction methods}\label{sec2.3}
Let $k$ be a positive integer. A subset $\mathcal{L}$ of the $k$--dimensional real vector space ${ \mathbb{R}^k}$ is called a lattice if there exists a basis $\{b_1, b_2, \ldots, b_k \}$ of $\mathbb{R}^k$ such that
\begin{align*}
	\mathcal{L} = \sum_{i=1}^{k} \mathbb{Z} b_i = \left\{ \sum_{i=1}^{k} r_i b_i \mid r_i \in \mathbb{Z} \right\}.
\end{align*}
We say that $b_1, b_2, \ldots, b_k$ form a basis for $\mathcal{L}$, or that they span $\mathcal{L}$. We
call $k$ the rank of $ \mathcal{L}$. The determinant $\text{det}(\mathcal{L})$, of $\mathcal{L}$ is defined by
\begin{align*}
	\text{det}(\mathcal{L}) = | \det(b_1, b_2, \ldots, b_k) |,
\end{align*}
with the $b_i$'s being written as column vectors. This is a positive real number that does not depend on the choice of the basis (see \cite{Cas}, Section 1.2).

Given linearly independent vectors $b_1, b_2, \ldots, b_k$ in $ \mathbb{R}^k$, we refer back to the Gram--Schmidt orthogonalization technique. This method allows us to inductively define vectors $b^*_i$ (with $1 \leq i \leq k$) and real coefficients $\mu_{i,j}$ (for $1 \leq j \leq i \leq k$). Specifically,
\begin{align*}
	b^*_i &= b_i - \sum_{j=1}^{i-1} \mu_{i,j} b^*_j,~~~
	\mu_{i,j} = \dfrac{\langle b_i, b^*_j\rangle }{\langle b^*_j, b^*_j\rangle},
\end{align*}
where \( \langle \cdot , \cdot \rangle \)  denotes the ordinary inner product on \( \mathbb{R}^k \). Notice that \( b^*_i \) is the orthogonal projection of \( b_i \) on the orthogonal complement of the span of \( b_1, \ldots, b_{i-1} \), and that \( \mathbb{R}b_i \) is orthogonal to the span of \( b^*_1, \ldots, b^*_{i-1} \) for \( 1 \leq i \leq k \). It follows that \( b^*_1, b^*_2, \ldots, b^*_k \) is an orthogonal basis of \( \mathbb{R}^k \). 
\begin{definition}
	The basis $b_1, b_2, \ldots, b_n$ for the lattice $\mathcal{L}$ is called reduced if
	\begin{align*}
		\| \mu_{i,j} \| &\leq \frac{1}{2}, \quad \text{for} \quad 1 \leq j < i \leq n,~~
		\text{and}\\
		\|b^*_{i}+\mu_{i,i-1} b^*_{i-1}\|^2 &\geq \frac{3}{4}\|b^*_{i-1}\|^2, \quad \text{for} \quad 1 < i \leq n,
	\end{align*}
	where $ \| \cdot \| $ denotes the ordinary Euclidean length. The constant $ {3}/{4}$ above is arbitrarily chosen, and may be replaced by any fixed real number $ y $ in the interval ${1}/{4} < y < 1$ (see \cite{LLL}, Section 1).
\end{definition}
Let $\mathcal{L}\subseteq\mathbb{R}^k$ be a $k-$dimensional lattice  with reduced basis $b_1,\ldots,b_k$ and denote by $B$ the matrix with columns $b_1,\ldots,b_k$. 
We define
\[
l\left( \mathcal{L},y\right)= \left\{ \begin{array}{c}
	\min_{x\in \mathcal{L}}||x-y|| \quad  ;~~ y\not\in \mathcal{L}\\
	\min_{0\ne x\in \mathcal{L}}||x|| \quad  ;~~ y\in \mathcal{L}
\end{array}
\right.,
\]
where $||\cdot||$ denotes the Euclidean norm on $\mathbb{R}^k$. It is well known that, by applying the
LLL--algorithm, it is possible to give in polynomial time a lower bound for $l\left( \mathcal{L},y\right)$, namely a positive constant $c_1$ such that $l\left(\mathcal{L},y\right)\ge c_1$ holds (see \cite{SMA}, Section V.4).
\begin{lemma}\label{lem2.5}
	Let $y\in\mathbb{R}^k$ and $z=B^{-1}y$ with $z=(z_1,\ldots,z_k)^T$. Furthermore, 
	\begin{enumerate}[(i)]
		\item if $y\not \in \mathcal{L}$, let $i_0$ be the largest index such that $z_{i_0}\ne 0$ and put $\sigma:=\{z_{i_0}\}$, where $\{\cdot\}$ denotes the distance to the nearest integer.
		\item if $y\in \mathcal{L}$, put $\sigma:=1$.
	\end{enumerate}
	\noindent Finally, let 
	\[
	c_2:=\max\limits_{1\le j\le k}\left\{\dfrac{||b_1||^2}{||b_j^*||^2}\right\}.
	\]
	Then, 
	\[
	l\left( \mathcal{L},y\right)^2\ge c_2^{-1}\sigma^2||b_1||^2:=c_1^2.
	\]
\end{lemma}
In our application, we are given real numbers $\eta_0,\eta_1,\ldots,\eta_k$ which are linearly independent over $\mathbb{Q}$ and two positive constants $c_3$ and $c_4$ such that 
\begin{align}\label{2.9}
	|\eta_0+x_1\eta_1+\cdots +x_k \eta_k|\le c_3 \exp(-c_4 H),
\end{align}
where the integers $x_i$ are bounded as $|x_i|\le X_i$ with $X_i$ given upper bounds for $1\le i\le k$. We write $X_0:=\max\limits_{1\le i\le k}\{X_i\}$. The basic idea in such a situation, due to \cite{Weg}, is to approximate the linear form \eqref{2.9} by an approximation lattice. So, we consider the lattice $\mathcal{L}$ generated by the columns of the matrix
$$ \mathcal{A}=\begin{pmatrix}
	1 & 0 &\ldots& 0 & 0 \\
	0 & 1 &\ldots& 0 & 0 \\
	\vdots & \vdots &\vdots& \vdots & \vdots \\
	0 & 0 &\ldots& 1 & 0 \\
	\lfloor C\eta_1\rfloor & \lfloor C\eta_2\rfloor&\ldots & \lfloor C\eta_{k-1}\rfloor& \lfloor C\eta_{k} \rfloor
\end{pmatrix} ,$$
where $C$ is a large constant usually of the size of about $X_0^k$ . Let us assume that we have an LLL--reduced basis $b_1,\ldots, b_k$ of $\mathcal{L}$ and that we have a lower bound $l\left(\mathcal{L},y\right)\ge c_1$ with $y:=(0,0,\ldots,-\lfloor C\eta_0\rfloor)$. Note that $ c_1$ can be computed by using the results of Lemma \ref{lem2.5}. Then, with these notations the following result  is Lemma VI.1 in \cite{SMA}.
\begin{lemma}[Lemma VI.1 in \cite{SMA}]\label{lem2.6}
	Let $S:=\displaystyle\sum_{i=1}^{k-1}X_i^2$ and $T:=\dfrac{1+\sum_{i=1}^{k}X_i}{2}$. If $c_1^2\ge T^2+S$, then inequality \eqref{2.9} implies that we either have $x_1=x_2=\cdots=x_{k-1}=0$ and $x_k=-\dfrac{\lfloor C\eta_0 \rfloor}{\lfloor C\eta_k \rfloor}$, or
	\[
	H\le \dfrac{1}{c_4}\left(\log(Cc_3)-\log\left(\sqrt{c_1^2-S}-T\right)\right).
	\]
\end{lemma}
Finally, we present an analytic argument which is Lemma 7 in \cite{GL}.  
\begin{lemma}[Lemma 7 in \cite{GL}]\label{Guz} If $ m \geq 1 $, $T > (4m^2)^m$ and $T > \displaystyle \frac{x}{(\log x)^m}$, then $$x < 2^m T (\log T)^m.$$	
\end{lemma}
SageMath 9.5 is used to perform all computations in this work.

\section{Proof of Theorem \ref{1.1}.}\label{Sec3}
In this section, we prove Theorem \ref{1.1}. To do this, we first state and prove some preliminary results. We start with the following.
\subsection{An upper bound on $n$ in terms of $s$ and $k$.}\label{subsec3.1}
\begin{lemma}\label{lem3.1}
Let $n \geq k + 1$ and $L_n^{(k)}=p_1^{\beta_1} \ldots p_s^{\beta_s} $, be the prime factorization of \( L_n^{(k)} \) with $\beta_i \geq 0$, for all $i=1,2,\ldots,s$. Then
	\[ \log n < 35s \log s + 3s \log k + 3 \log(12s + k). \]
	
\end{lemma}
\begin{proof}
We use Theorem \ref{thm:Mat} to get the inequality in Lemma \ref{lem3.1}. Because of our earlier deduction that $P(L_n^{(k)})\le 3$ for all $n\le k$, we can  assume that $n\ge k+1$. Moreover, the main result in \cite{Rih} tells us that the only $k$--generalized Lucas numbers that are powers of 2 are $L_0^{(k)}=2$, $L_1^{(k)}=1$ (for any $k\ge 2$), $L_3^{(2)}=2^2$ and $L_7^{(3)}=2^6$. So, we may further assume that $s\ge 2$ and $n\ge k+1\ge 3$.

Now, by the prime factorization of $L_n^{(k)}$ and \eqref{eq2.6}, we have 
\begin{align}\label{eq3.1}
	\left|p_1^{\beta_1} \ldots p_s^{\beta_s}-f_k(\alpha)(2\alpha-1)\alpha^{n-1}\right|<\dfrac{3}{2}.
\end{align}
Dividing both sides by $f_k(\alpha)(2\alpha-1)\alpha^{n-1}$, which is positive because $\alpha>1$, we get
\begin{align}\label{eq3.2}
	\left|p_1^{\beta_1} \ldots p_s^{\beta_s}\cdot(2\alpha-1)^{-1}\cdot\alpha^{-(n-1)}\cdot(f_k(\alpha))^{-1}-1\right|&<\dfrac{3}{2f_k(\alpha)(2\alpha-1)\alpha^{n-1}}\nonumber\\
	&<\dfrac{6}{(2\alpha-1)\alpha^{n-1}}\nonumber\\
	&<\dfrac{6}{\alpha^{n-1}},
\end{align}
where in the second inequality, we used relation \eqref{eq2.5}; i.e., $f_k(\alpha)>1/2$. Let 
$$
\Gamma=p_1^{\beta_1} \ldots p_s^{\beta_s}\cdot(2\alpha-1)^{-1}\cdot\alpha^{-(n-1)}\cdot(f_k(\alpha))^{-1}-1=e^{\Lambda}-1.
$$
Notice that $\Lambda\ne 0$, otherwise we would have
\begin{align}
\label{eq:case}
p_1^{\beta_1} \ldots p_s^{\beta_s} &= (2\alpha-1)\alpha^{n-1}f_k(\alpha)\nonumber\\
&=\dfrac{\alpha-1}{2+(k+1)(\alpha-2)}	(2\alpha-1)\alpha^{n-1}.
\end{align}
Conjugating the above relation by some automorphism of the Galois group of the splitting field of $\Psi_k (x)$ over $\mathbb{Q}$ which sends $\alpha$ to $\alpha_i$ for some $i>1$ and then taking absolute values, we get 
\begin{align}\label{eq3.3}
	p_1^{\beta_1} \ldots p_s^{\beta_s} 
	&=\left|\dfrac{\alpha_i-1}{2+(k+1)(\alpha_i-2)}	(2\alpha_i-1)\alpha_i^{n-1}\right|.
\end{align}
Note that from \eqref{eq3.3}, we have that $|2+(k+1)(\alpha_i-2)|\ge (k+1)|\alpha_i-2|-2>k-1$, as shown on page 1355 of \cite{Brl}. Hence, the right-hand side of \eqref{eq3.3} becomes
\begin{align*}
6 &=\min\{L_4^{(3)},L_4^{(2)}\}\le L_n^{(k)}=\left|\dfrac{\alpha_i-1}{2+(k+1)(\alpha_i-2)}(2\alpha_i-1)	\alpha_i^{n-1}\right| < \dfrac{|\alpha_i-1|\cdot|2\alpha_i-1|\cdot |\alpha_i|^{n-1}}{k-1}\le\dfrac{2\cdot 3\cdot 1}{k-1}\\
& <6,
\end{align*} 
for $k\ge 3$ (so, $n\ge k+1\ge 4$), or $k=2$ and $n\ge 4$, a contradiction. One can check directly that \eqref{eq:case} does not hold for the remaining case $k=2,~n=3$. So, $\Lambda\ne 0$. 

The algebraic number field containing the following $\gamma_i$'s is $\mathbb{K} := \mathbb{Q}(\alpha)$. We have $D = k$, $t := s+3$,
\begin{equation}\nonumber
	\begin{aligned}
		\gamma_{i}&:=p_i~~ \text{for}~i=1,2,\ldots, s, ~~~\gamma_{s+1}:=2\alpha-1, ~~~\gamma_{s+2}:=\alpha,~~~\gamma_{s+3}:=f_k(\alpha),\\
		b_{i}&:=\beta_i~~ \text{for}~i=1,2,\ldots, s, ~~~b_{s+1}:=-1, ~~~b_{s+2}:=-(n-1),~~~b_{s+3}:=-1.
	\end{aligned}
\end{equation}
Since $h(\gamma_{i})=\log p_i \le \log p_s$ for all $i=1,2,\ldots, s$, we take $A_i:=k \log p_s$ for all $i=1,2,\ldots, s$. Furthermore, $h(\gamma_{s+1}) <3/k$, for all $k\ge 2$, so we take $A_{s+1}:=3$. Additionally, $h(\gamma_{s+2})=(\log \alpha)/k <0.7/k$, so we take $A_{s+2}:=0.7$. Lastly, $h(\gamma_{s+3})<3\log k$ by relation \eqref{eq2.9}. Hence, we take $A_{s+3}:=3k\log k$.

Next, $B \geq \max\{|b_i|:i=1,2,\ldots, s,\ldots, s+3\}$. Notice that $b_i=\beta_i<n+1$, for all $i=1,2,\ldots, s$ by relation  \eqref{eq2.7}, so we take $B:=n+1$. Now, by Theorem \ref{thm:Mat},
\begin{align}\label{eq3.4}
	\log |\Gamma| &> -1.4\cdot 30^{s+6} \cdot(s+3)^{4.5}\cdot k^2 (1+\log k)(1+\log (n+1))\cdot (k\log p_s)^s\cdot 3\cdot 0.7\cdot 3k\log k\nonumber\\
	&> -1.4\cdot 30^{s}\cdot 30^6 \cdot s^{4.5}\cdot 3^{4.5}\cdot k^2 \cdot 3\log k\cdot 2\log (n+1)\cdot k^s(\log p_s)^s\cdot 3\cdot 0.7\cdot 3k\log k\nonumber\\
	&> -5.5\cdot 10^{12} \cdot 30^s s^{4.5}k^{3+s}(\log k)^2(\log p_s)^s \log (n+1).
\end{align}
Comparing \eqref{eq3.2} and \eqref{eq3.4}, we get
\begin{align*}
	(n-1)\log \alpha-\log 6&<5.5\cdot 10^{12} \cdot 30^s s^{4.5}k^{3+s}(\log k)^2(\log p_s)^s \log (n+1),
\end{align*}
which leads to
\begin{align*}
	n-1&<1.2\cdot 10^{13} \cdot 30^s s^{4.5}k^{3+s}(\log k)^2(\log p_s)^s \log (n+1),
\end{align*}
and adding $2$ to both sides yields
\begin{align}\label{eq3.5}
	n+1&<1.21\cdot 10^{13} \cdot 30^s s^{4.5}k^{3+s}(\log k)^2(\log p_s)^s \log (n+1).
\end{align}
Now, recall our assumption that $n\ge k+1$ implies $k\le n-1<n+1$. Moreover, the inequality $p_m < m^2$ holds for all $m\ge 2$. This is the Corollary to Theorem 3 on page 69 of \cite{Ros}. With these, inequality \eqref{eq3.5} becomes
\begin{align*}
	\dfrac{n+1}{(\log (n+1))^3}&<1.21\cdot 10^{13}  s^{4.5}k^{3+s}(60\log s)^s .
\end{align*}
We apply Lemma \ref{Guz} with $x:=n+1$, $m:=3$ and $T:=1.21\cdot 10^{13}  s^{4.5}k^{3+s}(60\log s)^s >(4m^2)^m=46656$. We get 
\begin{align*}
	n+1&<2^3\cdot 1.21\cdot 10^{13}  s^{4.5}k^{3+s}(60\log s)^s (\log (1.21\cdot 10^{13}  s^{4.5}k^{3+s}(60\log s)^s))^3\nonumber\\
	&= 9.68\cdot 10^{13}  s^{4.5}k^{3+s}(60\log s)^s \left(\log (1.21\cdot 10^{13})+4.5\log  s+(3+s)\log k+s\log(60\log s)\right)^3\nonumber\\
	&< 9.68\cdot 10^{13}  s^{4.5}k^{3+s}(60\log s)^s \cdot s^3\left(\dfrac{31}{s}+\dfrac{4.5}{s}\log  s+\left(1+\dfrac{3}{s}\right)\log k+\log(60\log s)\right)^3\nonumber\\
	&<9.68\cdot  10^{13}  s^{7.5}k^{3+s}(60\log s)^s (12s+k)^3,
\end{align*}
where we have used the fact that $(1+3/s)\log k<k$ for $k\ge 2$, $s\ge 2$ and 
$$(31/s)+(4.5/s)\log s+\log(60\log s)<12s\qquad {\text{\rm for}}\qquad s\ge 2.
$$
Therefore,
\begin{align*}
	n	&<9.7\cdot  10^{13}  s^{7.5}k^{3+s}(60\log s)^s (12s+k)^3,
\end{align*}
and hence
\begin{align*}
	\log n	&<\log (9.7\cdot  10^{13})+7.5\log s +(3+s)\log k +s\log(60\log s)+3\log (12s+k)\\
	&<33+7.5\log s +s\log(60\log s)+3s\log k+3\log (12s+k), ~~\text{since}~ 3+s<3s ~~\text{for all}~s\ge 2,\\
	&=s \log s \left(\dfrac{33}{s\log s}+\dfrac{7.5}{s}+\dfrac{\log(60\log s)}{\log s}\right)+ 3s \log k + 3 \log(12s + k),\\
	&<35s \log s + 3s \log k + 3 \log(12s + k).
\end{align*}
This completes the proof of Lemma \ref{lem3.1}.
\end{proof}
To proceed, observe that if $k\le s$, then Lemma \ref{lem3.1} implies that
\begin{align*}
	\log n	&<38s \log s  + 3 \log(13s)\\
	&<s\log s \left(38+\dfrac{8}{s\log s}+\dfrac{3}{s}\right)\\
	&<46s\log s,
\end{align*}
for $s\ge 2$. Using the well-known fact that $p_s>s\log s$, which is relation (3.12) from page 69 of \cite{Ros}, we have that
\begin{align*}
p_s>s\log s>\dfrac{1}{46}\log n.
\end{align*}
We therefore assume that $s<k$ for the remainder of this section. With this assumption, the conclusion of Lemma \ref{lem3.1} becomes
\begin{align}\label{eq3.6}
	\log n	&<38s \log k  + 3 \log(13k)\nonumber\\
	&<s\log k \left(38+\dfrac{8}{s\log k}+\dfrac{3}{s}\right)\nonumber\\
	&<46s\log k,
\end{align}
for $s\ge 2$ and $k\ge 2$. We proceed by distinguishing between two cases.
\subsection{The case $n\ge 2^{k/2}$}
Here, we have that 
\begin{align*}
\dfrac{k}{2}\log 2 \le \log n,	
\end{align*}
so that 
\begin{align*}
k \le \dfrac{2}{\log 2}\log n<\dfrac{2}{\log 2}\cdot 46s\log k<133s\log k.	
\end{align*}
From the above, we have $k/\log k < 133s$. We apply Lemma \ref{Guz} with the data: $x:=k$, $m:=1$ and $T:=133s>(4m^2)^m=4$, for all $s\ge 2$. We get
\begin{align}\label{eq3.7}
	k<2\cdot 133s\log(133s)=266s(\log 133+\log s)\le 266s\log s\left(\frac{\log 133}{\log 2}+1\right)<2143s\log s.
\end{align}
Therefore, 
\begin{align}\label{eq3.8}
	\log k<\log 2143+\log s+\log \log s<\left(\dfrac{\log 2143}{\log 2}+2\right)s\log s<14\log s,
\end{align}
holds for $s\ge 2$. Finally, we use Lemma \ref{lem3.1} again together with relations \eqref{eq3.7} and inequality \eqref{eq3.8} to conclude that 
\begin{align*}
	\log n&<  35s \log s + 3s\cdot 14 \log s + 3 \log(12s + 2143s\log s)<77s\log s+3 \log(2155s^2)\\
	& =77s\log s+6\log s+\log(2155) <s\log s\left(77+\frac{6}{s}+\frac{\log(2155)}{s\log s}\right)\\
	&<86s\log s,
\end{align*}
for $s\ge 2$. In the above, we used that $\log s<s$. Consequently, $p_s>s\log s>(1/86)\log n$ in this case.

\subsection{The case $n< 2^{k/2}$}\label{subsec3.3}
Let $\lambda > 0$ be such that $\alpha + \lambda = 2$. Since $2(1 - 2^{-k})<\alpha<2$, then we get that $\lambda < 2 - 2(1 - 2^{-k}) = 1/2^{k-1}$. That is, $\lambda \in (0, 1/2^{k-1})$. Moreover,	
	\begin{align*}
		\alpha^{n-1} &= (2 - \lambda)^{n-1} = 2^{n-1} \left(1 - \frac{\lambda}{2}\right)^{n-1} \\
		&= 2^{n-1}\left(e^{\log(1-\lambda/2)}\right)^{n-1}\\ 
		&\ge 2^{n-1}e^{-\lambda(n-1)}\\
		&\ge  2^{n-1}(1-\lambda(n-1)) ,
	\end{align*}
where we used the fact that $\log(1 - x) > -2x$ for all $x < 1/2$ and $e^{-x} \ge 1 - x$ for all $x \in \mathbb{R}$. 
	
Furthermore, 
$$\lambda(n - 1) < \dfrac{n - 1}{2^{k-1} }< \dfrac{2^{k/2}}{2^{k-1}} = 2^{k/2-1},$$
implying that $\alpha^{n-1} > 2^{n-1}(1 - 2^{k/2-1})$.
	It follows since $2(1 - 2^{-k})<\alpha<2$ that
	\[
	2^{n-1} - \frac{2^n}{2^{k/2}} < \alpha^{n-1} < 2^{n-1} <2^{n-1}+ \frac{2^n}{2^{k/2}},
	\]
	or
	\begin{equation} \label{eq3.9}
		\left| \alpha^{n-1} - 2^{n-1} \right| < \frac{2^n}{2^{k/2}}.
	\end{equation}
Next, consider the function $f_k(x)$ given at \eqref{eq:fk}. By the Mean-Value Theorem, there exists some $\omega \in (\alpha, 2)$ such that $f_k(\alpha) = f_k(2) + (\alpha - 2)f_k'(\omega)$. Observe that when $k \geq 2$, we obtain
	\begin{align*}
	|f_k'(\omega)| &= \dfrac{k-1}{(2 + (k + 1)(\omega - 2))^2 }\\
	&< \dfrac{k-1}{(2 - (k + 1)/2^{k-1})^2},~~\text{since}~\omega\in (\alpha,2)\subset (2-1/2^{k-1},2),\\
	&<k
	\end{align*}
	for $k\ge 3$, since $2^{k-1}\ge k+1$ for $k\ge 3$. It can be checked that the same holds for $k=2$. 
Hence,
	\begin{equation} \label{eq3.10}
		|f_k(\alpha) - f_k(2)| = |\alpha - 2||f'_k(\omega)| = \lambda |f'_k(\omega)| < k\lambda<\dfrac{k}{2^{k-1}}=\dfrac{2k}{2^{k}}.
	\end{equation}
Finally here, $2\alpha-1=2(2-\lambda)-1=3-2\lambda<3$ and 
\begin{align*}
	2\alpha-1=3-2\lambda>3-2\cdot \dfrac{1}{2^{k-1}}=3-\dfrac{4}{2^k},
\end{align*}
implying that 
\begin{align}\label{eq3.11}
	|(2\alpha-1)-3|<\dfrac{4}{2^k}.
\end{align}
From the above, if we write 
	\[
	\alpha^{n-1} = 2^{n-1} + \delta, \qquad f_k(\alpha) = f_k(2) + \eta\quad \text{and} \quad 2\alpha-1=3+\phi ,
	\]
then inequalities \eqref{eq3.9}, \eqref{eq3.10} and \eqref{eq3.11} become
	\begin{equation} \label{eq3.12}
		|\delta| < \frac{2^n}{2^{k/2}},\quad |\eta| < \frac{2k}{2^k} \quad \text{and} \quad |\phi|<\dfrac{4}{2^k} .
	\end{equation}
Moreover, since $f_k(2) = 1/2$ for all $k \geq 2$, we have
	\begin{align} \label{eq3.13}
		f_k(\alpha) \alpha^{n-1}(2\alpha-1) &=\left( f_k(2) + \eta\right) (2^{n-1} + \delta)(3+\phi)\nonumber\\
		& = \left(2^{n-2} + \frac{\delta}{2} + 2^{n-1}\eta + \eta \delta\right)(3+\phi)\nonumber\\
		&=3\cdot 2^{n-2} + \frac{3}{2}\delta + 3\cdot2^{n-1}\eta + 3\eta \delta+2^{n-2}\phi + \frac{\delta}{2}\phi + 2^{n-1}\eta\phi + \eta \delta\phi.
	\end{align}
Therefore, using \eqref{eq2.6} and relations \eqref{eq3.12} and \eqref{eq3.13}, we get
	\begin{align*}
		\left|p_1^{\beta_1} \cdots p_s^{\beta_s} -3\cdot2^{n-2}\right| &= \left| \left(L^{(k)}_n - f_k(\alpha)\alpha^{n-1}(2\alpha-1)\right) +\left( \frac{3}{2}\delta + 3\cdot2^{n-1}\eta + 3\eta \delta+2^{n-2}\phi + \frac{\delta}{2}\phi + 2^{n-1}\eta\phi + \eta \delta\phi\right) \right| \\
		&< \frac{3}{2}+ \dfrac{3\cdot2^{n-1}}{2^{k/2}} + \frac{3\cdot2^{n}k}{2^{k}} + \frac{3\cdot2^{n+1}k}{2^{3k/2}}+\dfrac{2^n}{2^k}+\dfrac{2^{n+1}}{2^{3k/2}}+\dfrac{2^{n+2}k}{2^{2k}}+\dfrac{2^{n+3}k}{2^{5k/2}}\\
		&< 3\cdot2^{n-2}\left(\frac{1}{2^{n-1}}+ \dfrac{2}{2^{k/2}} + \frac{4k}{2^{k}} + \frac{8k}{2^{3k/2}}+\dfrac{4/3}{2^k}+\dfrac{8/3}{2^{3k/2}}+\dfrac{16k/3}{2^{2k}}+\dfrac{32k/3}{2^{5k/2}}\right)\\
		&<3\cdot2^{n-2}\left(\frac{1}{2^{k/2}}+ \dfrac{2}{2^{k/2}} + \frac{5}{2^{k/2}} + \frac{4}{2^{k/2}}+\dfrac{2}{2^{k/2}}+\dfrac{3}{2^{k/2}}+\dfrac{8}{2^{k/2}}+\dfrac{11}{2^{k/2}}\right)\\
		&=3\cdot2^{n-2}\frac{36}{2^{k/2}}.
	\end{align*}
In the above, we used that $k\le 2^{k-1}$ and that $k<(5/4) 2^{k/2}$ for $k\ge 2$.  Dividing both sides above by $3\cdot 2^{n-2}$, having in mind that $p_1=2$ and $p_2=3$, we get
	\begin{align} \label{eq3.14}
		\left|p_1^{\beta_1-n+2}\cdot p_2^{\beta_2-1} \cdots p_s^{\beta_s} -1\right|< \frac{36}{2^{k/2}}.
	\end{align} 
Now, we intend to apply Theorem \ref{thm:Mat} on the left-hand side of \eqref{eq3.14}. Let $$\Gamma_1=p_1^{\beta_1-n+2}\cdot p_2^{\beta_2-1} \cdots p_s^{\beta_s}-1=e^{\Lambda_1}-1.$$
Notice that $\Lambda_1\ne 0$, otherwise we would have
$L_n^{(k)}=p_1^{\beta_1} \ldots p_s^{\beta_s} = 3\cdot 2^{n-2}$ but since $n\ge k+1$, this contradicts \eqref{eq:32}. Here, $t := s$,
\begin{equation}\nonumber
	\begin{aligned}
		\gamma_{i}&:=p_i~~ \text{for}~i=1,2,\ldots, s,~~~
	b_1:=\beta_1-n+2,~~~b_2:=\beta_2-1~~~	b_{i}&:=\beta_i~~ \text{for}~i=3,\ldots, s.
	\end{aligned}
\end{equation}
The algebraic number field containing $\gamma_i$'s is $\mathbb{K} := \mathbb{Q}$, so we  take $D = 1$.
Since $h(\gamma_{i})=\log p_i \le \log p_s$ for all $i=1,2,\ldots, s$, we take $A_i:= \log p_s$ for all $i=1,2,\ldots, s$.

Again, $B \geq \max\{|b_i|:i=1,2,\ldots, s\}$. Notice that $b_i=\beta_i<n+1$, for all $i=1,2,\ldots, s$ by relation  \eqref{eq2.7}, so we take $B:=n+1$. Now, by Theorem \ref{thm:Mat},
\begin{align}\label{eq3.15}
	\log |\Gamma_1| &> -1.4\cdot 30^{s+3} \cdot s^{4.5}\cdot 1^2 (1+\log 1)(1+\log (n+1))\cdot (\log p_s)^s\nonumber\\
	&> -1.4\cdot 30^{s}\cdot 30^3 \cdot s^{4.5}\cdot 2\log 2 n\cdot (\log p_s)^s\nonumber\\
	&> -10^5 \cdot 30^s s^{4.5}\log n\cdot(\log p_s)^s .
\end{align}
Comparing \eqref{eq3.14} and \eqref{eq3.15}, we get
\begin{align*}
	\dfrac{k}{2}\log 2-\log 36 &<10^5 \cdot 30^s s^{4.5}\log n\cdot(\log p_s)^s\\
	k&<3\cdot 10^5 \cdot  s^{4.5}\log n\cdot(60\log s)^s, ~~\text{since}~p_s<s^2,\\
	&<3\cdot 10^5 \cdot  s^{4.5}(46s\log k)\cdot(60\log s)^s,
\end{align*}
since $\log n<46s\log k$ in \eqref{eq3.6}. Therefore, we can write
\begin{align*}
	\dfrac{k}{\log k}<1.4 \cdot 10^7 s^{5.5}\cdot(60\log s)^s.
\end{align*}
We again apply Lemma \ref{Guz} with the data: $x:=k$, $m:=1$ and $T:=7 \cdot 10^6 s^{5.5}\cdot(60\log s)^s>(4m^2)^m=4$, for $s\ge 2$. We get 
\begin{align*}
	k&<2^1\cdot 1.4 \cdot 10^7 s^{5.5}\cdot(60\log s)^s (\log (1.4 \cdot 10^7 s^{5.5}\cdot(60\log s)^s)^1\nonumber\\
	&= 2.8\cdot 10^{7}  s^{5.5}(60\log s)^s \left(\log (1.4\cdot 10^{7})+5.5\log  s+s\log(60\log s)\right)\nonumber\\
	&=2.8\cdot 10^{7}  s^{5.5}(60\log s)^s\cdot s\log s \left(\dfrac{17}{s\log s}+\dfrac{5.5}{s}+\dfrac{\log 60}{\log s}+\dfrac{\log\log s}{\log s}\right)\nonumber\\
	&<6\cdot 10^{8}s^{6.5}(60\log s)^s\log s.
\end{align*}
As a result, 
\begin{align}\label{eq3.16}
	\log k &<\log 6\cdot 10^{8}+6.5\log s+s\log(60\log s)+\log\log s\nonumber\\
	&< 21+6.5\log s+s\log60+s\log\log s+\log\log s\nonumber\\
	&=s\log s\left(\dfrac{21}{s\log s}+\dfrac{6.5}{s}+\dfrac{\log 60}{\log s}+\dfrac{\log \log s}{\log s}+\dfrac{\log \log s}{s\log s}\right)\nonumber\\
	&<26s\log s.
\end{align}
To finish the proof, recall we are treating the case when $n < 2^{k/2}$ , therefore $\log n <(k/2) \log 2 < k$. This and relation \eqref{eq3.16} tell us that $\log \log n < \log k < 26 s \log s$, hence $p_s > s \log s > (1/26) \log \log n$. This completes the proof of Theorem \ref{1.1}.

\section{Proof of Theorem \ref{1.2}}\label{Sec4}
We proceed in a way similar as in Section \ref{Sec3}. Specifically, we prove following estimates.
\begin{lemma}\label{lem4.1}
If $P (L_n^{(k)} ) \le 7$, then:
\begin{enumerate}[(a)]
	\item The inequality	$$n  < 1.4 \cdot 10^{27} k^7 (\log k)^3,$$ 	holds for $k \ge 2$ and $n \ge 3$.
	\item If $k > 1000$, then 
	\begin{center}
		$k < 1.64\cdot 10^{20}$	and
	$~n <4.6\cdot 10^{173}$.
	\end{center}
\end{enumerate}
\end{lemma}
\begin{proof}~
\begin{enumerate}[(a)]
	\item To prove the first part, we use the same arguments used in Subsection \ref{subsec3.1}. Indeed, for $s=4$, we have from inequality \eqref{eq3.2} that
	\begin{align}\label{eq4.1}
		\left|2^{a}\cdot 3^b\cdot 5^c\cdot 7^d\cdot (2\alpha-1)^{-1}\cdot\alpha^{-(n-1)}\cdot(f_k(\alpha))^{-1}-1\right|
		&<\dfrac{6}{\alpha^{n-1}},
	\end{align}
for which we obtain as before, by substituting $s=4$ in \eqref{eq3.5}, that 	
	\begin{align*}
		n+1&<1.21\cdot 10^{13} \cdot 30^4 \cdot 4^{4.5}k^{3+4}(\log k)^2(\log 7)^4 \log (n+1)\\
		&<7.2\cdot 10^{24}k^7(\log k)^2 \log (n+1).
	\end{align*}
In particular, we have 
	\begin{align*}
	\dfrac{n+1}{\log(n+1)}
	&<7.2\cdot 10^{24}k^7(\log k)^2 .
\end{align*}
We now apply Lemma \ref{Guz} with  $x:=n+1$, $m:=1$, $T:=7.2\cdot 10^{24}k^7(\log k)^2 >(4m^2)^m=4$ for $k\ge 2$. We get 
\begin{align*}
	n+1&<2\cdot 7.2\cdot 10^{24}k^7(\log k)^2 \log (7.2\cdot 10^{24}k^7(\log k)^2)\\
	&=1.44\cdot 10^{25}k^7(\log k)^2 \left(\log (7.2\cdot 10^{24})+7\log k+2\log\log k\right)\\
	&=1.44\cdot 10^{25}k^7(\log k)^3 \left(\dfrac{58}{\log k}+7+\dfrac{2\log\log k}{\log k}\right)\\
	&<1.4\cdot 10^{27}k^7(\log k)^3.
\end{align*}
\item In the second part, if $k>1000$, then 
\begin{align*}
	n	<1.4\cdot 10^{27}k^7(\log k)^3<2^{k/2}.
\end{align*}
For this reason, we can use the same arguments from Subsection \ref{subsec3.3} relation \eqref{eq3.14} to write 
	\begin{align} \label{eq4.2}
	\left|2^{a-n+2}\cdot 3^{b-1}\cdot 5^c\cdot 7^d -1\right|< \frac{36}{2^{k/2}},
\end{align}
from which after applying Matveev's result with $s=4$ as in Subsection \ref{subsec3.3} we get
\begin{align*}
	k<3\cdot 10^{8}\cdot 4^{6.5}(60\log 4)^4\log 4<1.64\cdot 10^{20}.
\end{align*}
Lastly, we substitute this upper bound on $k$ in part (a) of Lemma \ref{lem4.1} to get
\begin{align*}
	n  &< 1.4 \cdot 10^{27}(1.64\cdot 10^{20})^7 (\log 1.64\cdot 10^{20})^3<4.6\cdot 10^{173}.
\end{align*}
\end{enumerate}
This completes the proof of Lemma \ref{lem4.1}.
\end{proof}
To complete the proof of Theorem \ref{1.2}, we proceed in two cases, that is, the case $k\le1000$ and the case $k>1000$. We use similar analyses given on pages 1363 and 1364 of \cite{Brl}.

\subsection{The case $k\le 1000$.}
In this subsection, we treat the cases when $k \in [2, 1000]$. Note that when $k\le 1000$, then $n< 4.62\cdot 10^{50}$ by Lemma \ref{lem4.1}. The next step is to reduce this large upper bound on $n$. To do this, we let
\[
\tau_1 := a \log 2 + b \log 3 + c \log 5 + d \log 7-\log (2\alpha-1) - (n - 1) \log \alpha - \log f_k(\alpha),
\]
so that \eqref{eq4.1} can be rewritten as
\begin{equation}\label{eq4.3}
	|e^{\tau_1} - 1| < \frac{6}{\alpha^{n-1}}.
\end{equation}
Observe that $\tau_1 \neq 0$. Moreover, if $\tau_1 > 0$, then $e^{\tau_1} - 1 > 0$, so from \eqref{eq4.3} we obtain
\[
0 < \tau_1 < \frac{6}{\alpha^{n-1}},
\]
where we used the fact that $x \leq e^x - 1$ for all $x \in \mathbb{R}$. Next, we treat the case $\tau_1 < 0$. Note that if $n \geq 7$, then $6/\alpha^{n-1} < 1/2$ for all $k \geq 2$. Thus, from \eqref{eq4.3}, we get that $|e^{\tau_1} - 1| < 1/2$ or $e^{\tau_1} < 2$. Since $\tau_1 < 0$, we obtain
\[
0 < |\tau_1| \leq e^{|\tau_1|} - 1 = e^{|\tau_1|}|e^{\tau_1} - 1| < \frac{12}{\alpha^{n-1}}.
\]
Thus, in all cases, the inequality
\begin{equation}
	|\tau_1| < \frac{12}{\alpha^{n-1}}
\end{equation}
holds for $k \geq 2$ and $n \geq 7$. Observe that $|\tau_1|$ is an expression of the form
\[
|x_1 \log 2 + x_2 \log 3 + x_3 \log 5 + x_4 \log 7 +x_5 \log (2\alpha-1)+ x_6 \log \alpha + x_7 \log f_k(\alpha)|,
\]
where $x_1 := a$, $x_2 := b$, $x_3 := c$, $x_4 := d$, $x_5 := -1$, $x_6 := -(n - 1)$, $x_7 := -1$ are integers with 
\[
\max\{|x_i| : 1 \leq i \leq 7\}  < n<1.4 \cdot 10^{27} k^7 (\log k)^3,
\]
where we used Lemma \ref{lem4.1}.

For each $k \in [2, 1000]$, we used the LLL--algorithm to compute a lower bound for the smallest nonzero number of the form $|\tau_1|$, with integer coefficients $x_i$ not exceeding $1.4 \cdot 10^{27} k^7 (\log k)^3$ in absolute value. Specifically, we consider the approximation lattice
$$ \mathcal{A}=\begin{pmatrix}
	1 & 0 & 0 & 0 & 0 & 0 & 0 \\
	0 & 1 & 0 & 0 & 0 & 0 & 0 \\
	0 & 0 & 1 & 0 & 0 & 0 & 0 \\
	0 & 0 & 0 & 1 & 0 & 0 & 0 \\
	0 & 0 & 0 & 0 & 1 & 0 & 0 \\
	0 & 0 & 0 & 0 & 0 & 1 & 0 \\
	\lfloor C\log 2\rfloor & \lfloor C\log 3\rfloor& \lfloor C\log 5 \rfloor& \lfloor C\log 7 \rfloor & \lfloor C\log(2\alpha-1) \rfloor & \lfloor C\log\alpha \rfloor & \lfloor C\log f_k(\alpha) \rfloor
\end{pmatrix} ,$$
with $C:= 10^{355}$ and choose $y:=\left(0,0,0,0,0,0,0\right)$. Now, by Lemma \ref{lem2.5}, we get $$l\left(\mathcal{L},y\right)^2>c_1^2=5.96\cdot10^{106}.$$
So, Lemma \ref{lem2.6} gives $S=1.5\cdot 10^{102}$ and $T=1.7\cdot 10^{51}$. Since $c_1^2\ge T^2+S$, then choosing $c_3:=12$ and $c_4:=\log\alpha$, we get $n-1\le 1448$.

Finally, we wrote a simple program in SageMath (see Appendix 1), to look at $k$-generalized Lucas numbers for $2 \leq k \leq 1000$ and $k + 1 \leq n \leq 1449$. Instead of factoring the numbers fully, we checked if they could be divided by $2$, $3$, $5$, and $7$ until we couldn't divide them anymore. This way, we found out if each number could be written using only these primes. The numbers we got are the ones given in Theorem \ref{1.2}. This completes the analysis in the case $k \in [2, 1000]$. 

\subsection{The case $k>1000$}
Lastly, we treat the case when $k > 1000$. At this point, we need to reduce our absolute upper bound on $k$, see Lemma \ref{lem4.1}, by using again the LLL--algorithm described in Lemma \ref{lem2.6}. To do this, let
\[
\tau_2 := (a - n + 2) \log 2 + (b-1) \log 3 + c \log 5 + d \log 7,
\]
so that we can rewrite \eqref{eq4.2} as
\begin{equation}\label{eq4.5}
	|e^{\tau_2} - 1| < \frac{36}{2^{k/2}}.
\end{equation}
Again, it is clear that $\tau_2 \neq 0$. If $\tau_1 > 0$, then $e^{\tau_1} - 1 > 0$, so from \eqref{eq4.5} we obtain
\[
0 < \tau_2 <  \frac{36}{2^{k/2}},
\]
by similar arguments as before. If $\tau_2 < 0$, then we can note from \eqref{eq4.5} that $36/2^{k/2} < 1/2$ for all $k>1000$. Hence, it follows from \eqref{eq4.5} that $|e^{\tau_2} - 1| < 1/2$ which implies $e^{\tau_2} < 2$. Since $\tau_2 < 0$, we obtain that
\begin{equation*}
	0 < |\tau_2| \leq e^{|\tau_2|} - 1 = e^{|\tau_2|}|e^{\tau_2} - 1| < \frac{72}{2^{k/2}}.
\end{equation*}
Thus, in all cases, we have 
\begin{equation}
	 |\tau_2| < \frac{72}{2^{k/2}}.
\end{equation}
Like before, observe that $|\tau_2|$ is an expression of the form
\[
|x_1 \log 2 + x_2 \log 3 + x_3 \log 5 + x_4 \log 7|,
\]
where $x_1 := a - n + 2$, $x_2 := b-1$, $x_3 := c$, $x_4 := d$. From the second part of Lemma \ref{lem4.1}, we have
\[
\max\{|x_i| : 1 \leq i \leq 4\} < n < 4.6 \times 10^{173}.
\]
At this point, we consider the approximation lattice
$$ \mathcal{A}=\begin{pmatrix}
	1 & 0 & 0 & 0  \\
	0 & 1 & 0 & 0  \\
	0 & 0 & 1 & 0  \\
	
	\lfloor C\log 2\rfloor & \lfloor C\log 3\rfloor& \lfloor C\log 5 \rfloor& \lfloor C\log 7 \rfloor 
\end{pmatrix} ,$$
with $C:= 10^{695}$ and choose $y:=\left(0,0,0,0\right)$. By Lemma \ref{lem2.5}, we get $$l\left(\mathcal{L},y\right)^2>c_1^2=10^{350}.$$
So, Lemma \ref{lem2.6} gives $S=8.5\cdot 10^{347}$ and $T=9.2\cdot 10^{173}$. Since $c_1^2\ge T^2+S$, then choosing $c_3:=72$ and $c_4:=\log 2$, we get $k/2\le 1733$. This implies that $k\le 3466$ and the first part of Lemma \ref{lem4.1} tells us that $n < 4.6 \cdot 10^{54}$.

With this new upper bound for $n$ we repeat the LLL--algorithm once again to get a lower bound of $|\tau_2|$, where now the coefficients $x_i$ are integers satisfying 
\[
\max\{|x_i|: 1 \leq i \leq 4\} < n < 4.6 \cdot 10^{54}.
\]
With the same approximation lattice and $C:=10^{220}$, we get $c_1^2=10^{111}$, $S=8.464\cdot 10^{109}$ and $T=9.2\cdot 10^{54}$. We then obtain that $k \leq 1106$. After repeating this process 2 more times, we finally find that $k <1000$, which is a contradiction. Thus, Theorem \ref{1.2} is proved. \qed

\section*{Acknowledgments} 
The first author thanks the Eastern Africa Universities Mathematics Programme (EAUMP) for funding his doctoral studies. The second author worked on this paper during a fellowship at STIAS in the second part of 2023.
This author thanks STIAS for hospitality and support.

\section*{Addresses}
$ ^{1} $ Department of Mathematics, School of Physical Sciences, College of Natural Sciences, Makerere University, Kampala, Uganda

 Email: \url{hbatte91@gmail.com}\\

\vspace{0.35cm}
\noindent 
$ ^{2} $ School of Mathematics, Wits University, Johannesburg, South Africa  

Email: \url{Florian.Luca@wits.ac.za}

\section*{Appendices}
\subsection*{Appendix 1}\label{app1}
\begin{verbatim}
	# Define a memoization dictionary to store previously computed values
	memo = {}
	
	def k_generalized_lucas_iterative(n, k):
	# Base cases
	if n == 0:
	return 2
	elif n == 1:
	return 1
	elif n < 2 - k:
	return 0
	
	# Check if we have already computed the value
	if (n, k) in memo:
	return memo[(n, k)]
	
	# Initialize a list with the base cases
	lucas_nums = [0] * (2 - k) + [2, 1] + [None] * (n - 1)
	
	# Compute the k-generalized Lucas numbers iteratively
	for i in range(2, n + 1):
	lucas_nums[i] = sum(lucas_nums[i - j] for j in range(1, k + 1) if i - j >= 0)
	# Store the computed number in the memo dictionary
	memo[(i, k)] = lucas_nums[i]
	
	return lucas_nums[n]
	
	def is_of_form_2a_3b_5c_7d(lucas_val):
	for prime in [2, 3, 5, 7]:
	while lucas_val % prime == 0:
	lucas_val //= prime
	return lucas_val == 1
	
	# Check for solutions
	for k in range(2, 1001):
	for n in range(k + 1, 1450):
	lucas_val = k_generalized_lucas_iterative(n, k)
	if is_of_form_2a_3b_5c_7d(lucas_val):
	print(f"For k={k}, n={n}: L_n^{(k)} = {lucas_val}")
\end{verbatim}

\end{document}